\documentclass[a4paper,11pt,reqno]{amsart}

\usepackage{enumerate}
\usepackage{amssymb, amsmath,amsthm}
\usepackage{mathrsfs}
\usepackage{amscd}
\usepackage[active]{srcltx}
\usepackage{verbatim}
\usepackage{a4wide}
\usepackage[pdfborder={0 0 0}]{hyperref}  
\usepackage{graphicx}
\usepackage{bbm}

\theoremstyle{plain}
\newtheorem{theorem}{Theorem}[section]

\newtheorem{lemma}[theorem]{Lemma}
\newtheorem{proposition}[theorem]{Proposition}

\newtheorem{definition}[theorem]{Definition}

\newtheorem*{theorem*}{Theorem}
\newtheorem{propdef}[theorem]{Proposition/Definition}

\theoremstyle{remark}
\newtheorem{remark}[theorem]{Remark}
\newtheorem{example}[theorem]{Example}

\numberwithin{equation}{section}

\def\N{{\mathbf N}}
\def\Z{{\mathbf Z}}
\def\R{{\mathbf R}}

\newcommand{\cM}{{\mathcal M}}

\newcommand{\eps}{\varepsilon}
\renewcommand{\phi}{\varphi}

\newcommand{\ip}[1]{\langle {#1}\rangle}
\DeclareMathOperator{\Lip}{Lip}
\newcommand{\Lipn}{\Lip_N}
\DeclareMathOperator{\diam}{diam}
\newcommand{\beq}{\begin{equation}}
\newcommand{\eeq}{\end{equation}}
\newcommand{\bal}{\begin{aligned}}
\newcommand{\eal}{\end{aligned}}
\newcommand{\ben}{\begin{enumerate}}
\newcommand{\beni} {\begin{enumerate}[(i)]}
\newcommand{\een}{\end{enumerate}}
\newcommand{\bit}{\begin{itemize}}
\newcommand{\eit}{\end{itemize}}
\newcommand{\beqw}{\begin{equation*}}
\newcommand{\eeqw}{\end{equation*}}
\newcommand{\bex}{\begin{example}}
\newcommand{\eex}{\end{example}}
\newcommand{\bre}{\begin{example}}
\newcommand{\ere}{\end{example}}
\newcommand{\bma}{\begin{bmatrix}}
\newcommand{\ema}{\end{bmatrix}}

\newcommand{\one}{{{\bf 1}}}

\newcommand{\hrho}{\hat\rho}

\newcommand{\dd}{\; \mathrm{d}}
\newcommand{\ddt}{\frac{\mathrm{d}}{\mathrm{d}t}}

\newcommand{\cP}{\mathscr{P}}

\renewcommand{\aa}{\mathbf{a}}
\newcommand{\bb}{\mathbf{b}}
\newcommand{\cc}{\mathbf{c}}
\newcommand{\ei}{\mathbf{e}_i}

\newcommand{\LN}{\mathbf{T}_N^d}

\newcommand{\cX}{\mathcal{X}}

\newcommand{\cEN}{\mathcal{E}_N}
\newcommand{\cW}{\mathcal{W}}
\newcommand{\cWN}{\mathcal{W}_N}
\newcommand{\tWN}{\widetilde{\mathcal{W}}_N}

\newcommand{\PL}{\mathscr{P}({\LN})}
\newcommand{\PX}{\mathscr{P}(\mathcal{X})}
\newcommand{\dist}{\mathsf{d}_N}

\newcommand{\cA}{\mathcal{A}}

\newcommand{\I}{\mathbf{T}^d}
\newcommand{\QA}{Q_{\aa}^N}
\newcommand{\RAp}{R_{\aa,i+}^N}

\newcommand{\RAm}{R_{\aa,i-}^N}
\newcommand{\RApm}{R_{\aa,i\pm}^N}

\newcommand{\RN}{\mathscr{R}^N}

\newcommand{\PN}{\mathcal{P}_N}
\newcommand{\QN}{\mathcal{Q}_N}

\newcommand{\tx}{\tilde{x}}
\newcommand{\ty}{\tilde{y}}

\newcommand{\rhon}{\rho^N}
\newcommand{\Vn}{V^N}

\newcommand\prob{\mathscr P}
\newcommand\D{\mathscr P}

\newcommand{\sfh}{{\sf H}}
\newcommand{\sfhh}{{\sf h}}

\usepackage[notref,notcite,color,final]{showkeys}

\begin{document}

\title[Convergence of discrete transportation metrics]
{Gromov-Hausdorff convergence of discrete transportation metrics}

\author{Nicola Gigli}
\address{
Universit\'e de Nice-Sophia Antipolis\\ 
Labo. J.-A. Dieudonn\'e\\ 
UMR 6621, Parc Valrose\\
06108 Nice Cedex 02\\
France}
\email{gigli@unice.fr}

\author{Jan Maas}
\address{
Institute for Applied Mathematics\\
University of Bonn\\
Endenicher Allee 60\\
53115 Bonn\\
Germany}
\email{maas@uni-bonn.de}

\thanks{JM acknowledges support by Rubicon subsidy 680-50-0901 of the Netherlands Organisation for Scientific Research (NWO)}

\begin{abstract}
This paper continues the investigation of `Wasserstein-like' transportation distances for probability measures on discrete sets. We prove that the discrete transportation metrics $\cW_N$ on the $d$-dimensional discrete torus $\LN$ with mesh size $\frac1N$ converge, when $N\to\infty$, to the standard 2-Wasserstein distance $W_2$ on the continuous torus in the sense of Gromov--Hausdorff. 
This is the first convergence result for the recently developed  discrete transportation metrics $\cW$. The result shows the compatibility between these metrics and the well-established $2$-Wasserstein metric.
\end{abstract}

\date\today

\maketitle

\tableofcontents

\section{Introduction}

In recent years, the theory of optimal transportation has drawn a lot of attention in the mathematical community, see for instance the monograph \cite{Vil09} and references therein. A crucial role in this context is played by the quadratic transportation distance $W_2$, known as Wasserstein distance: it is a distance between probability measures on a metric space particularly well-suited to study measure dynamics, with important applications in the fields of functional and geometric inequalities, parabolic PDEs and other areas (see \cite{Vil09}).

It turns out that when $\cX$ is a discrete space, there are no non-constant Lipschitz curves in the $2$-Wasserstein space $(\PX,W_2)$, hence $W_2$ is not the right metric to deal with when studying problems where the evolution of measures on discrete spaces is involved.

Motivated by this remark, various authors \cite{CHLZ11,Ma11,Mie11a} proposed a definition of a variant of the distance $W_2$, denoted by $\cW$, on the set of probability measures over a finite set $\cX$ endowed with a Markov kernel $K$. The Markov kernel encodes the geometric information of the space, and the distance $\cW$ is defined via an appropriate variant of the Benamou-Brenier formula. It turns out that the non-existence of Lipschitz curves, and in particular geodesics, is circumvented with the use of  $\cW$. Moreover, this distance has several of the properties that $W_2$ has in the continuous setting, e.g.,  it can be used to study evolution problems \cite{CHLZ11,Ma11,Mie11a} and to give a definition of lower Ricci curvature bounds \cite{EM11,Mie11b}.

Although the definition of $\mathcal W$ formally resembles that of $W_2$  given by the Benamou-Brenier formulation of the optimal transport problem \cite{BB00}, up to now there was no explicit link between the two metrics. The purpose of this paper is to bridge this gap by proving a Gromov-Hausdorff convergence result in an important special case, which we believe may serve as guideline to prove similar results in geometrically more complicated  situations.

Specifically,  we consider the space $\cP(\I)$ of probability measures on the torus $\I:=\R^d/\Z^d$, endowed with the usual $2$-Wasserstein metric $W_2$.
We also consider the $d$-dimensional periodic lattice $\LN:=(\Z / N\Z)^d$ with mesh size $\frac1N$, and endow the space of probability measures $\prob(\LN)$ with its renormalised discrete transportation metric $\cWN$ as defined in \cite{CHLZ11,Ma11,Mie11a} (see Section \ref{sec:prel} below). Our main result reads then as follows:

\begin{theorem}\label{thm:main}
Let $d \geq 1$. Then the metric spaces $(\prob(\LN),\mathcal W_N)$ converge to $(\prob(\I),W_2)$ in the sense of Gromov-Hausdorff as $N \to \infty$.
\end{theorem}

In order not to make this introduction too long, we refer to the body of the paper for the precise definitions of the distances involved, see in particular Section \ref{se:discmetr}. The outline of the strategy of the proof is in  Section \ref{subsec:ingred},  the crucial estimates needed in our argument are contained in Section \ref{subsec:estim}, and then the proof is completed in Section \ref{subsec:wrapup}.

\medskip
For the sake of comparison, let us mention that if $(\cX_N,d_N)$ is a sequence of compact metric spaces converging in the GH-sense to a limit space, then the corresponding $2$-Wasserstein spaces also converge in GH-sense, as is easy to prove (see, e.g., Theorem 28.6 in \cite{Vil09}). The crucial point in Theorem \ref{thm:main} is that the discrete transportation metric $\cW$ is used instead of the $2$-Wasserstein metric. This makes the result non-trivial, and it allows for potential applications to convergence of gradient flows \cite{CG11,Ser11}, since GH-convergence results have proven to be powerful in this context \cite{Gi10}.
 
Different results linking discretisations of the Wasserstein distance, evolution equations and passage to the limit can be found in, e.g., \cite{GT06,PR12}.
Convergence results for lower Ricci curvature bounds on discrete spaces have been obtained in \cite{BS09}. Note however, that the notion of discrete Ricci curvature in that paper is based on the usual $2$-Wasserstein metric. A different notion of Ricci curvature has been studied in \cite{EM11}. The latter notion relies on the metric $\cW$, which is the main object in the present paper.

\subsection*{Acknowledgement}

{\small
This work has been started during a visit of the second named author to the University of Nice. He thanks this institution for its kind hospitality and support. The authors thank the anonymous referees for their careful reading and useful suggestions.
}

\section{Preliminaries}\label{sec:prel}

\subsection{The \texorpdfstring{$2$}{2}-Wasserstein metric} Let $\cM$ be a compact smooth Riemannian manifold and $\prob (M)$ the set of Borel probability measures on it. The Wasserstein distance $W_2$ on $\prob (M)$ is usually defined by minimizing the transport cost with respect to the cost function distance-squared. It has been emphasized by Benamou and Brenier \cite{BB00} that a completely different introduction to the subject can be given in terms of solutions to the continuity equation. The following result has been proved for $M=\R^d$ in \cite{AGS08} (see also \cite{O01}), the case of general manifolds being a consequence of Nash's embedding theorem (see also \cite[Proposition 2.5]{Erb10} for a direct proof on manifolds).
\begin{propdef}\label{prop:BeBr}
Let $\cM$ be a compact smooth Riemannian manifold and $\mu,\nu\in\prob (M)$. Then we have
\begin{equation}
\label{eq:BeBr}
W_2^2(\mu,\nu)=\min \int_0^1\int_M |v_t|^2(x)\dd \mu_t(x)\dd t\;,
\end{equation}
the minimum being taken among all distributional solutions $(\mu_t,v_t)$ of the continuity equation
\begin{equation}
\label{eq:conteq}
\ddt\mu_t+\nabla\cdot(v_t\mu_t)=0\;,
\end{equation}
such that $t\mapsto\mu_t$ is weakly continuous in duality with $C(M)$ and $\mu_0=\mu$, $\mu_1=\nu$.
\end{propdef}
In the sequel, when considering the continuous setting we will work with $M$ being the $d$-dimensional torus $\I:=\R^d/\Z^d$ and we will consider solutions to the continuity equation in terms of probability densities and momentum vector fields. To fix the ideas, we give the following definition.
\begin{definition}[Solutions to the continuity equation in the continuous torus]\label{def:solcont}
Consider the mappings $[0,1]\times\I\ni(t,x)\mapsto \rho_t(x)\in\R$ and $[0,1]\times \I\mapsto V_t(x)\in\R^d$. We say that $(\rho_t,V_t)$ solves the continuity equation
\begin{equation}
\label{eq:contcont}
\ddt\rho_t+\nabla \cdot V_t=0\;,
\end{equation}
provided both $(t,x)\mapsto \rho_t(x)$ and $(t,x)\mapsto V_t(x)$ are in $L^1([0,1]\times \I)$, $t\mapsto\rho_t$ is continuous with respect to convergence in duality with $C(\I)$, and \eqref{eq:contcont} is satisfied in the sense of distributions.
\end{definition}

\subsection{Discrete transportation metrics}

In several recent works \cite{CHLZ11,Ma11,Mie11a} discrete analogues of $W_2$ have been considered, which are well suited to study evolution equations in a discrete setting. The definition of the Wasserstein distance requires a metric on the underlying space. In \cite{Ma11}, instead, the starting point is a  Markov kernel $K$ on the finite set $\cX$, i.e., we assume that $K : \cX \times \cX \to \R_+$ satisfies 
$
 \sum_{y \in \cX} K(x,y) = 1
$
for all $x \in \cX$.
We assume that $K$ is irreducible and denote the unique steady state by $\pi$. Thus $\pi$ is  the unique probability measure on $\cX$ satisfying 
\begin{align*}
 \pi(y) = \sum_{x \in \cX} \pi(x) K(x,y)
\end{align*}
for all $y \in \cX$.
We shall assume that $K$ is reversible, i.e., the detailed balance equations
\begin{align*}
 K(x,y) \pi(x) = K(y,x) \pi(y)
\end{align*}
hold for all $x,y \in \cX$. Since basic Markov chain theory implies that $\pi$ is strictly positive, we can -- and will -- identify probability measures on $\cX$ with their densities with respect to $\pi$, i.e., we set
\[
\PX := \Big\{ \, \rho : \cX \to \R_+ \ | \ \sum_{x \in \cX} \pi(x) \rho(x)  = 1 \, \Big\}\;.
\]
In order to define the metric $\cW$ on $\PX$, it is necessary to fix a function $\theta : \R_+ \times \R_+ \to \R_+$. Various choices have been considered in \cite{EM11,Ma11}, but here we will focus on the case where $\theta$ is the logarithmic  mean, which is defined by
\begin{align*}
 \theta(s,t) = \int_0^1 s^{1-p} t^p \dd p\;.
\end{align*}
With this choice of $\theta$, it has been shown in \cite{CHLZ11,Ma11,Mie11a} that the discrete heat flow is the gradient flow of the Boltzmann-Shannon entropy with respect to $\cW$.
For $\rho \in \PX$ and $x, y \in \cX$ we set
\begin{align*}
 \hrho(x,y) = \theta(\rho(x), \rho(y))\;,
\end{align*}
which can be regarded informally as being ``the density $\rho$ at the edge $(x,y)$''. 
According to \cite[Lemma 2.9]{EM11}, the following definition can be taken as one of the equivalent definitions of the transportation metric $\cW$ on $\PX$ associated to the logarithmic mean.
\begin{definition}\label{def:W}
Let $K$ be an irreducible and reversible Markov kernel on a finite set $\cX$, and let $\bar\rho_0,\bar\rho_1\in\PX$. The distance $\cW(\bar\rho_0,\bar\rho_1)$ is defined by
\begin{align}\label{eq:WV}
\cW(\bar\rho_0, \bar\rho_1)^2 =   \inf\bigg\{ \frac12 \int_0^1\sum_{x,y \in \cX}\frac{V_t(x,y)^2}{\hrho_t(x,y)} K(x,y) \pi(x) \dd t  \bigg\}\;,
\end{align}
where the infimum runs over all curves $[0,1]\ni t\mapsto (\rho_t,V_t)$ such that:
\begin{enumerate}[(i)]
\item $\rho_t\in \PX$ for any $t\in[0,1]$, the function $t\mapsto \rho_t(x)$ is continuous for any $x\in\cX$, and  $\rho_0=\bar\rho_0$, $\rho_1=\bar\rho_1$;
\item $V_t:\cX \times \cX \to\R$ for any $t\in[0,1]$, and the function $t\mapsto V_t(x,y)$ belongs to  $L^1(0,1)$ for any $x,y\in\cX$;
\item the ``discrete continuity equation''
\begin{equation}
\label{eq:discce}
\displaystyle{\ddt\rho_t(x) + \frac12 \sum_{y \in \cX}\big(V_t(x,y) - V_t(y,x)\big) K(x,y) = 0}
\end{equation}
holds for all $x \in \cX$  in the sense of distributions.
\end{enumerate}
\end{definition}
\subsection{The transportation metric on the discrete torus}\label{se:discmetr}

In this paper we shall only be concerned with simple random walk on the $d$-dimensional discrete torus $\LN := (\Z / N\Z)^d  = \{0, \ldots, N-1\}^d$, in which case the kernel $K_N : \LN \times \LN \to [0,1]$ is given by 
\begin{align*}
 K_N(\aa, \bb) =  \left\{ \begin{array}{ll}
 \frac1{2d},
  & \text{$\bb = \aa\pm \ei \mod N$ for some $i \in \{1, \ldots,  d\}$}\;,\\
 0,
  & \text{otherwise}\;.\end{array} \right.
\end{align*}
Here, $\ei$ denotes the $i$-th unit vector. All computations in $\LN$ will be performed modulo $N$ without further mentioning.
 
In this case the stationary probability measure $\pi_N$ is the uniform measure given by $\pi_N(\aa) = N^{-d}$ for all $\aa \in \LN$. Therefore, the collection of probability densities with respect to $\pi_N$ is given by
\begin{align*}
  \D(\LN) = 
  \Big\{ \, \rho_N : \LN \to \R_+ \ \Big| \ 
    \sum_{\aa \in \LN} \rho_N(\aa)  = N^d \, \Big\}\;.
\end{align*}
For functions $f, g : \LN \to \R$ we consider the normalized $L^2$-inner product 
\begin{align*}
 \ip{f,g}_{L^2_N}  & = \frac{1}{N^d}\sum_{\aa \in \LN} f(\aa)g(\aa)
 \end{align*}
and the Dirichlet form
\begin{align*}
 \cEN(f,g) & = \frac{1}{N^{d-2}} \sum_{\aa \in \LN} \sum_{i = i}^d 
 \big( f(\aa + \ei) - f(\aa) \big)\big( g(\aa + \ei) - g(\aa) \big)\;.
\end{align*}
Furthermore we set
\begin{align*}
 \| f \|_{L^2_N}  & = \sqrt{ \ip{f,f}_{L^2_N}}\;,\qquad
 \cEN(f)  = \cEN(f,f)\;.
\end{align*}
Let $\Delta_N$ be the discrete Laplacian, defined by
\begin{align*}
 \Delta_N f(\aa) &=2 d N^2 (K_N - I) f(\aa)= N^2 \sum_{i=1}^d \Big( f(\aa + \ei) - 2f(\aa)  +  f(\aa - \ei)\Big)
\end{align*}
for $\aa \in \LN$. Notice that following integration by parts formula holds:
\begin{align}\label{eq:ibp}
 \cEN(f,g) = - \ip{\Delta_N f, g}_{L^2_N}\;.
\end{align}
Moreover, given $g:\LN\to\R$, the equation $\Delta_N f=g$ can be solved if and only if $\sum_{\aa\in\LN}g(\aa)=0$, in which case the solution is unique. 
We shall use the well-known Poincar\'e inequality on $\LN$, which we now recall.
\begin{proposition}[Poincar\'e inequality on $\LN$]\label{prop:Poincare}
Let $d \geq 1$ and $N \geq 4$. For all $f : \LN \to \R$ with $\sum_{\aa \in \LN} f(\aa) = 0$ we have
\[
\begin{split}
\| f \|_{L^2_N}^2 &\leq \frac{1}{2N^2(1- \cos(2\pi/N))} \cEN(f)\;,\\
\cEN(\Delta_N^{-1}f)&\leq \frac{1}{2N^2(1- \cos(2\pi/N))}\| f \|_{L^2_N}^2\;.
\end{split}
\]
\end{proposition}
\begin{proof}
One way to prove the first inequality is as follows. If $d = 1$, then the spectrum of the operator $I - K_N$ on $L^2(\LN,\pi_N)$ consists of the eigenvalues
\begin{align*}
{1- \cos(2\pi n/N)}\;, \qquad 0 \leq n \leq N-1\;,
\end{align*}
(see, e.g., \cite[Section 4.2]{DSC96}),
which yields the result if $d = 1$.
The result in dimension $d > 1$ follows by tensorization (see, e.g., \cite[Lemma 3.2]{DSC96}). 

The second inequality follows from the first one, using the integration by parts formula \eqref{eq:ibp}.
\end{proof}
\begin{remark}
In the limit $N \to \infty$, one recovers the classical Poincar\'e inequality on the torus $\I$:
\begin{align*}
    \| f \|_{L^2(\I)}^2 \leq \frac{1}{2\pi^2}\|\nabla f\|_{L^2(\I)}^2\;,
\end{align*}
valid for any $f$ with zero mean.\end{remark}

It will be useful to introduce some more notation.
For $\aa = (a_1, \ldots, a_d)\in \LN$ we define the cube $\QA$ by
\begin{align*}
 \QA := \Big[\frac{a_1}{N}, \frac{a_1 + 1}{N}\Big) \times \cdots \times
 \Big[\frac{a_d}{N}, \frac{a_d + 1}{N}\Big) 
 \subseteq \I\;,
\end{align*}
so that the torus $\I = \R^d / \Z^d$ can be written as the disjoint union
\begin{align*}
  \I =  \bigcup_{\aa \in \LN} \QA\;.
\end{align*}
For $i = 1, \ldots, d$, the facets of $\QA$ will be denoted by
\begin{align*}
 \RAm &=  \Big[\frac{a_1}{N}, \frac{a_1 + 1}{N}\Big] \times \,  \ \cdots \ 
 \Big \{ \frac{a_i}{N} \Big\}   \  \cdots \ \,  \times
 \Big[\frac{a_d}{N}, \frac{a_d + 1}{N}\Big]\;, \\
 \RAp &=  \Big[\frac{a_1}{N}, \frac{a_1 + 1}{N}\Big] \times \cdots
  \Big\{ \frac{a_i + 1}{N}\Big \}  \cdots \times
 \Big[\frac{a_d}{N}, \frac{a_d + 1}{N}\Big]\;,
\end{align*}
see Figure 1.

\begin{figure}\label{fig:1}
\scalebox{.37}
{\includegraphics{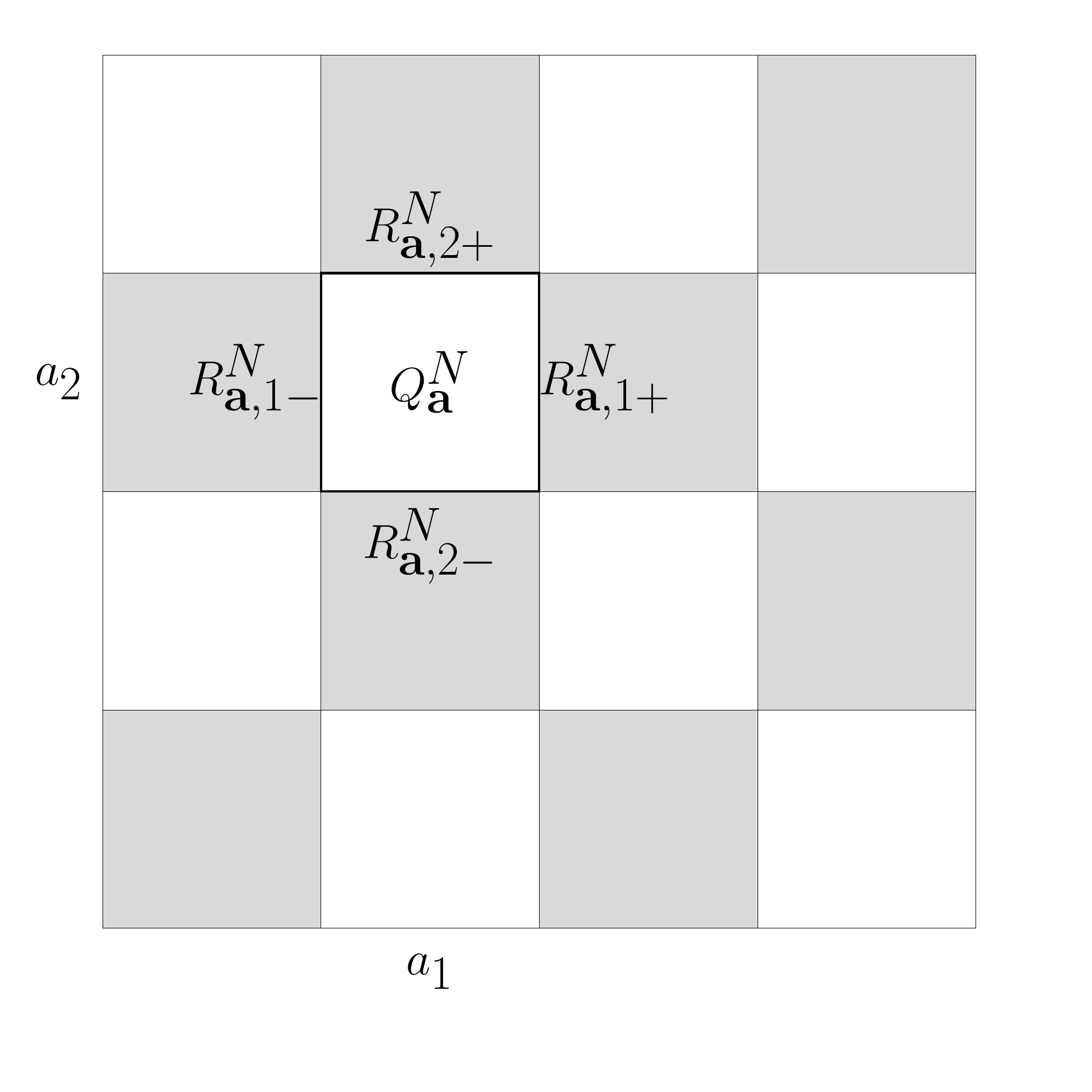}}
\caption{The cube $Q_\aa$ is drawn with its facets in $\mathbf{T}_N^d$ with $d = 2$, $N = 4$, and $\aa = (2,3)$.} 
\end{figure}

The collection of all these facets $\RApm$ will be denoted by $\RN$. 
For $\rho \in \PL$ and $R = \RApm \in \RN$ we shall write 
\begin{align*}
 \hrho_N(R) := \theta( \rho_N(\aa), \rho_N(\aa \pm \ei) )\;.
\end{align*}
Notice that $K_N(\aa,\bb)$ is non-zero only for $\aa,\bb$ such that $\aa-\bb=\pm\ei$ for some $i=1,\ldots, d$. Moreover, if $V$ satisfies the discrete continuity equation \eqref{eq:discce}, then the same holds for its anti-symmetrisation $V^{\rm asym}$ defined by $V_t^{\rm asym}(\aa, \bb) = \frac12(V_t(\aa, \bb) - V_t(\bb, \aa))$, and we have
\begin{align*}
 \sum_{x,y \in \cX}\frac{V_t^{\rm asym}(x,y)^2}{\hrho_t(x,y)} K(x,y) \pi(x)
 \leq
 \sum_{x,y \in \cX}\frac{V_t(x,y)^2}{\hrho_t(x,y)} K(x,y) \pi(x)\;.
\end{align*}
Therefore, in Definition \ref{def:W}(ii) it suffices to consider vector fields $V :  \LN \times \LN \to \R$ which are anti-symmetric, i.e.,  $V(\aa, \bb) = -V(\bb,\aa)$. This will be our convention from now on. Moreover, we shall identify an antisymmetric vector field $V$ with a function $V : \RN \to \R$ defined by
\begin{align*}
 V(\RAp) := V(\aa, \aa+\ei)\;.
\end{align*}

Let $\cW_{K_N}$ denote the metric on $\PL$ associated with the kernel $K_N$ according to Definition \ref{def:W}. It will be convenient to work with the normalised metric
\begin{align*}
  \cWN := \frac{\cW_{K_N}}{N\sqrt{2d}}\;,
\end{align*}
which is a quantity of order $1$. 

Given a probability density $\rho_N\in \D(\LN)$ and a `momentum vector field' $V_N:\RN\to\R$, the \emph{action} $\cA_N$ of $(\rho_N,V_N)$ is defined by
\begin{equation}
\label{eq:defaction}
\cA_N(\rho_N,V_N):= \frac{1}{4d^2 N^{d+2}}\sum_{R\in\RN}\frac{V_N(R)^2}{\hrho_N(R)}\;.
\end{equation}
With this notation and taking Definition \ref{def:W} into account, it is immediate to obtain the following expression for the metric $\cWN$.

\begin{lemma}\label{lem:WN-formula}
For any $\bar\rho_{N,0}, \bar\rho_{N,1} \in \D(\LN)$ we have
\begin{equation}
\label{eq:defwn}
\cWN(\bar\rho_{N,0}, \bar\rho_{N,1})^2= \inf \bigg\{ \int_0^1 \cA_N(\rho_{N,t},V_{N,t})\dd t\bigg\}\;,
\end{equation}
where the infimum runs over all  curves $[0,1]\ni t\mapsto (\rho_{N,t},V_{N,t})$ such that:
\begin{enumerate}[(i)]
\item $\rho_{N,t}\in \D(\LN)$ for any $t\in[0,1]$, and the function $t\mapsto \rho_{N,t}(\aa)$ is continuous for any $\aa\in\LN$ with $\rho_{N,0}=\bar\rho_{N,0}$, $\rho_{N,1}=\bar\rho_{N,1}$;
\item $V_{N,t}:\RN\to\R$ for any $t\in[0,1]$, and the function $t\mapsto V_{N,t}(R)$ belongs to  $L^1(0,1)$ for any $R\in\RN$;
\item the discrete continuity equation
\begin{equation}
\label{eq:contn}
 \displaystyle\ddt \rho_{N,t}(\aa)+ \frac{1}{2d}\sum_{i=1}^d\Big( V_{N,t}(\RAp)-  V_{N,t}(\RAm)\Big)= 0
 \end{equation}
holds for all $\aa \in \LN$ in the sense of distributions.
\end{enumerate}
\end{lemma}

By analogy with Definition \ref{def:solcont} we formulate the following discrete counterpart.
\begin{definition}[Solutions to the continuity equation in the discrete torus]
Let $[0,1]\times\LN\ni (t,\aa)\mapsto\rho_{N,t}(\aa)\in\R$ and $[0,1]\times\RN\ni (t,R)\mapsto V_{N,t}(R)\in\R^d$. We say that $(\rho_{N,t},V_{N,t})$ is a solution to the discrete continuity equation \eqref{eq:contn} provided that $(i)$, $(ii)$ and $(iii)$ in Lemma \ref{lem:WN-formula} are fulfilled.
\end{definition}
Finally, we recall a couple of properties of $\cWN$ that will be used in the sequel. We shall use the metric $\dist$ on $\LN$ defined by
\begin{align*}
\dist(\aa,\bb) = \frac1N\sqrt{\sum_{i=1}^d |a_i-b_i|^2}
\end{align*}
for $\aa, \bb \in \LN$. Recall that the computations are understood modulo $N$.
We let 
\begin{align}\label{eq:W2N}
W_{2,N}
\end{align}
denote the standard $2$-Wasserstein distance on $\prob(\LN)$ induced by the  distance $\dist$ on $\LN$. 
In the following result we collect some basic properties of the metric $\cWN$.

\begin{proposition}\label{prop:uniform} 
The following assertions hold.
\begin{enumerate}[(i)]
\item The function $(\rho, \sigma) \mapsto \cWN^2(\rho,\sigma)$ is convex on $\D(\LN) \times \D(\LN)$ with respect to linear interpolation. 
\item 
There exists a universal constant $c> 0$ such that
\[
\cWN\leq {c}{\sqrt{d}}\, W_{2,N}\;.
\]
In particular, the diameter of the spaces $(\D(\LN), \cWN)$ is bounded by a constant depending only on the dimension.
\end{enumerate}
\end{proposition}
\begin{proof}
The first assertion has been proved in \cite[Proposition 2.11]{EM11}.
 For the second assertion, we apply \cite[Proposition 2.14]{EM11} to obtain
\[
\cWN\leq \frac{c}{N} W_{2,N}'\;,
\]
where $c \approx 1,56$ is a universal constant and  $W_{2,N}'$ is the  $2$-Wasserstein distance on $\prob(\LN)$ induced by the graph distance $d_N'$ on $\LN$, defined by $d_N'(\aa,\bb):=\sum_i|\aa_i-\bb_i|$.
Since $d_N'(\aa,\bb)\leq \sqrt{d}N \dist(\aa,\bb)$, we have
$W_{2,N}' \leq \sqrt{d}N W_{2,N}$, which implies the desired estimate. 
Since the diameter of the spaces $(\LN, \dist)$ is uniformly bounded by a dimensional constant, the same holds for the spaces $(\prob(\LN),W_{2,N})$, and the final assertion follows as well.
\end{proof}

\subsection{Some properties of the heat semigroup on the discrete and continuous torus}

We endow the continuous torus $\I$ with its natural Riemannian flat distance, and we denote the Lebesgue measure by $\pi$.

Let $({\sfh}_{t})_{t \geq 0}$ be the heat semigroup on $\I$ with generator $\Delta$, acting either on measures or functions. The heat semigroup on $\LN$ is the semigroup generated by the discrete Laplacian $\Delta_N$, and  will be denoted by $({\sfh}^{N}_{t})_{t \geq 0}$.

Let $\sfhh_t$ be the heat kernel on $\I$, i.e., the density of ${\sfh}_t(\delta_0)$ with respect to $\pi$. Similarly, ${\sfhh}_t^{N}$ will denote the heat kernel on $\LN$, which is defined by 
${\sfhh}^{N}_t(x) = {\sfh}^{N}_{t}(N^d \one_{\{ \mathbf{0} \}} )(x)$. 
We thus have the formulas 
\begin{align*}
{\sfh}_tf(x)=\int_{\I}  {\sfhh}_t(x-y)f(y)\dd\pi(y)\;, \qquad 
{\sfh}^{N}_{t}f_N(\aa)= \frac{1}{N^d}\sum_{\bb \in \LN} {\sfhh}_{t}^N(\aa-\bb)f_N(\bb)\;,
\end{align*}
 valid for all $L^1$-functions $f : \I \to \R$ and $f_N: \LN \to \R$.

The heat semigroup on $\I$ acts on vector fields as well coordinatewise. Similarly, the action of ${\sfh}^{N}_{t}$ on a vector field $V_N:\RN\to \R$ can be defined via
\begin{align}\label{eq:heat-vector}
{\sfh}^{N}_{t}V_N(\RAp):=
\frac{1}{N^d}
\sum_{\bb\in\LN}{\sfhh}_{t}^N(\aa-\bb)V_N(R^N_{\bb,i+})\;.
\end{align}

Given a function $f:\I\to\R$, its Lipschitz constant will be denoted by $\Lip(f)$. Similarly, we define the Lipschitz constant of a function $f:\LN\to\R$  by
\begin{align*}
\Lipn(f):=\sup_{\aa \neq \bb} 
  \frac{|f(\aa) - f(\bb)|}{\dist(\aa,\bb)}\;.
\end{align*}
The propositions below collect some basic properties of the heat flows that we will use in the sequel.

\begin{proposition}[Heat flow on the continuous torus]\label{prop:heatcont}
The following assertions hold for all $s > 0$.
\begin{enumerate}[(i)]
\item There exist constants $c(s)> 0$ and $C(s) < \infty$ such that for any $\mu\in\prob(\I)$ the density $\rho_s$ of ${\sfh}_s \mu$ satisfies 
\begin{align*}
\rho_s\geq c(s)\qquad \text{and}\qquad \Lip(\rho_s) \leq C(s)\;.
\end{align*}
Furthermore, there exists a dimensional constant $C < \infty$ such that
\begin{align*}
W_2({\sfh}_s \mu, \mu) \leq C \sqrt{s}\;.
\end{align*}
\item There exists a constant $C(s)< \infty$ such that for any $f \in L^1(\I)$ we have
\begin{align*}
\|{\sfh}_s f\|_{L^\infty}  + \Lip({\sfh}_s f)
    \leq C(s) \|f\|_{L^1} \;.
\end{align*}

\item Let $(\mu_t)\subset\prob(\I)$ be a geodesic, let $v_t$ be the corresponding velocity vector fields achieving the minimum in \eqref{eq:BeBr}, and let $\rho_{s,t}$ and $V_{s,t}$ be the densities of ${\sfh}_s(\mu_t)$ and ${\sfh}_s(v_t\mu_t)$ respectively. Then, $t\mapsto(\rho_{s,t},V_{s,t})$ is a solution to the continuity equation \eqref{eq:contcont}, and we have
\begin{equation}
\label{eq:ricci0}
\int_0^1\int_{\I}\frac{V_{s,t}^2(x)}{\rho_{s,t}(x)}\dd x\dd t\leq W_2^2(\mu_0,\mu_1)\;.
\end{equation}
\end{enumerate}
\end{proposition}
\begin{proof}
The first assertions in $(i)$, with $c(s) = \inf_{x \in \I} \sfhh_s(x)$ and $C(s) = {\Lip}(\sfhh_s)$, are easily deduced from the representation of the heat semigroup as a convolution semigroup. The same method can be used to prove $(ii)$. To prove the last claim in $(i)$, notice that by the convexity of $W_2^2$ it is sufficient to prove the claim when $\mu$ is a Dirac mass. In this case the result follows from the fact that the heat kernel on the torus can be represented by periodization of the heat kernel on $\R^d$, and the parabolic scaling of the latter.

Finally, $(iii)$ follows from the convexity of $\R^d\times\R^+\ni(x,a)\mapsto \frac{x^2}a$ and the fact that  ${\sfh}_s$ is a convolution operator, see, e.g., Lemma 8.1.10 in \cite{AGS08}.
\end{proof}

\begin{proposition}[Heat flow on the discrete torus]\label{prop:heatdisc} The following assertions hold for $s > 0$. 
\begin{enumerate}[(i)]
\item There exists a constant $C(s) < \infty$ depending only on $s > 0$ and the dimension $d$, such that for any $\rho_N\in\D(\LN)$ we have 
\begin{align*}
\Lip_N({\sfh}_s^N \rho_N)\leq \min\big\{{C}(s),\ \Lip_N(\rho_N)\big\}\;.
\end{align*}
\item For any $\rho_N\in \D(\LN)$ and any momentum vector field $V_N:\RN\to\R^d$ we have
\[
\cA_N({\sfh}^{N}_{s}\rho_N,{\sfh}^{N}_{s}V_N)\leq\cA_N(\rho_N,V_N)\;.
\]
\end{enumerate}
\end{proposition}
\begin{proof}
The estimate $\Lip_N({\sfh}_s^N \rho_N)\leq  \Lip_N(\rho_N)$ in $(i)$ is a simple consequence of the fact that the heat semigroup consists of convolution operators. Taking the convexity of $(x,a,b)\mapsto\frac{x^2}{\theta(a,b)}$ into account, this also gives $(ii)$. 

To prove the remaining bound in $(i)$, we note that for any probability density $\rho_N \in \PL$,
\begin{align*}
  | {\sfh}^N_s \rho_N(\aa) - {\sfh}^N_s \rho_N(\bb) |
   &  = 
    \frac{1}{N^d} 
    \bigg| \sum_{\cc \in \LN}
      \Big( {\sfhh}_s^N(\aa - \cc) - {\sfhh}_s^N(\bb - \cc) \Big)\rho_N(\cc) \bigg|
 \\&  \leq
    \frac{1}{N^d} 
      \Bigg(\sum_{\cc \in \LN} \rho_N(\cc)\bigg)
      \sup_{\cc \in \LN} \big| {\sfhh}_s^N(\aa - \cc) - {\sfhh}_s^N(\bb - \cc) \big|
 \\&  =
	  \sup_{\cc \in \LN} \big| {\sfhh}_s^N(\aa - \cc) - {\sfhh}_s^N(\bb - \cc) \big|\;.
\end{align*}
Since ${\sfhh}^{N}_{s}(\aa) = {\sfhh}^{1,N}_{s}(a_1)\cdot \ldots \cdot {\sfhh}^{1,N}_{s}(a_d)$, where ${\sfhh}^{1,N}$ denotes the heat kernel in one dimension, we infer that 
\begin{align*}
 \big| {\sfhh}_s^N(\aa) - {\sfhh}_s^N(\bb) \big|
 & \leq  \| \sfhh_s^{1,N}\|_{L^\infty}^{d-1}
   \sum_{k=1}^d |\sfhh_s^{1,N}(a_k) - \sfhh_s^{1,N}(b_k)|
\\&  \leq \sqrt{d} \dist(\aa,\bb)\;
   \| \sfhh_s^{1,N}\|_{L^\infty}^{d-1} \Lipn(\sfhh_s^{1,N})\;,
\end{align*}
and therefore
\begin{align}\label{eq:1dremains}
  \Lipn( {\sfh}^N_s \rho_N)
   \leq\sqrt{d}\;
   \| \sfhh_s^{1,N}\|_{L^\infty}^{d-1} \Lipn(\sfhh_s^{1,N})\;,
\end{align}
so it remains to obtain bounds on the heat kernel in one dimension. These can be obtained using the well-known (and easy to check) fact that, if $d =1$, the spectrum of the operator $-\Delta_N$ consists of the eigenvalues 
\begin{align*}
 \lambda_\ell = 2N^2\big(1 - \cos({2\pi \ell}/{N})\big)\;, \qquad
  \ell \in L_N := 
   \bigg\{ z \in \Z \ : \ \Big\lfloor-\frac{N}{2}\Big\rfloor + 1 \leq z \leq
    \Big\lfloor\frac{N}{2}\Big\rfloor    \bigg\}\;.
\end{align*}
Note that $\lambda_\ell = \lambda_{-\ell}$. The corresponding eigenvectors $v_\ell$ are given by
\begin{align*}
 v_\ell(\aa) =  \exp\Big( \frac{2\pi i \ell \aa}{N} \Big)\;, \qquad \ell \in L_N\;.
\end{align*}
As a consequence, the heat kernel $\sfhh_s^{1,N}$ can be written explicitly as 
\begin{align*}
\sfhh_s^{1,N}(\aa) = \sum_{\ell \in L_N} e^{-\lambda_\ell s} v_\ell(\aa)\;.
\end{align*}
We shall use the fact that there exist constants $c> 0$ and $\tilde c < \infty$ such that for all $N \geq 1$ and $\ell \in L_N$,
\begin{align*}
 |\lambda_\ell| \geq  c \ell^2\;,\qquad
 \|v_\ell\|_{L^\infty} 
 \leq  1\;,\qquad \text{and} \qquad
 \Lip_N(v_\ell) \leq \tilde c |\ell| \;.
\end{align*}
It follows that for some constant $C > 0$ and all $\aa, \bb \in \LN$,
\begin{align*}
  \big| \sfhh_s^{1,N}(\aa)\big|
  &  \leq \sum_{\ell \in L_N} e^{-\lambda_\ell s} | v_\ell(\aa) |
  \leq \sum_{\ell \in \Z} e^{- c \ell^2 s } 
  \leq C \Big( 1 + \frac{1}{\sqrt{s}}\Big)\;,\\
  \big| \sfhh_s^{1,N}(\aa) - \sfhh_s^{1,N}(\bb) \big|
  &  \leq \sum_{\ell \in L_N} e^{-\lambda_\ell s} | v_\ell(\aa) - v_\ell(\bb) |
  \leq C \sum_{\ell \in \Z} \ell e^{- c \ell^2 s }  \dist(\aa,\bb)
  \leq \frac{C}{s}  \dist(\aa,\bb)\;,
\end{align*}
so that $\|\sfhh_s^{1,N}\|_{L^\infty} \leq C (1 + s^{-1/2})$ and 
$\Lip_N(\sfhh_s^{1,N}) \leq C s^{-1}$. Plugging these estimates into \eqref{eq:1dremains}, we obtain the desired result.
\end{proof}

\section{Proof of the main result}
\label{sec:proof}

\subsection{Ingredients and structure of the proof}
\label{subsec:ingred}

In order to prove the stated Gromov-Hausdorff convergence of the spaces $(\cP(\LN), \cWN)$, we will introduce the natural mappings from the continuous torus to the discrete one, and those going the other way around.

First we construct discrete measures by integration over cubes, and discrete vector fields by integration over facets:

\begin{definition}[From $\I$ to $\LN$]\label{def:PN}
Given a probability measure $\mu\in\prob(\I)$ and $N\in\N$ the probability density $\PN(\mu)\in\prob(\LN)$ is defined as
\[
\PN(\mu)(\aa):=N^d\mu({\QA})\;.
\]
Similarly, given a continuous momentum vector field $V = (V_1, \ldots, V_d) : \I \to \R^d$ we define $\PN (V): \RN \to \R$ by
\begin{align*}
 \PN (V)(R) := 2d N^d \int_{R} V_i(x)\dd x\;, \qquad  R = \RApm \in \RN\;.
\end{align*}
\end{definition}

Probability densities on $\I$ are defined by piecewise constant extensions of densities on $\LN$, and  vector fields on $\I$ are defined by linear interpolation.

\begin{definition}[From $\LN$ to $\I$]\label{def:QN}
Given a probability density $\rho^N\in\prob{(\LN)}$ and a momentum vector field $V^N:\RN\to\R$, the probability measure $\QN(\rho^N)\pi\in \prob(\I)$ and the momentum vector field $\QN(V^N):\I\to\R^d$ are defined as
\begin{align*}
\QN(\rhon)(x) &:= \rhon(\aa)\;,\\
 \QN(\Vn)_i(x) &:= \frac{1}{2dN}\Big( 
   (1 - N x_i + a_i) \Vn(\RAm)
   +    (N x_i - a_i) \Vn(\RAp)
 \Big)\;,
\end{align*}
where $\aa = (a_1, \ldots, a_d)\in \LN$ is uniquely determined by the condition $x = (x_1, \ldots, x_d) \in \QA$.
\end{definition}

The maps $\PN$, $\QN$ will be the ones that we use to prove Gromov-Hausdorff convergence. They are constructed in such a way that ensures that solutions of the continuity equation are mapped to solutions of the continuity equation.
\begin{proposition}\label{prop:solcont}
The following assertions hold:
\begin{enumerate}
\item 
Let $(\rho_t,V_t)$ be a solution to the continuity equation \eqref{eq:contcont} such that the mapping $x\mapsto V_t(x)$ is continuous for almost every $t$. Then $(\PN(\rho_t),\PN(V_t))$ solves the discrete continuity equation \eqref{eq:contn}.
\item 
Vice versa, let $(\rho_{N,t},V_{N,t})$ be a solution to the discrete continuity equation \eqref{eq:contn}. Then $(\QN(\rho_{N,t}),\QN(V_{N,t}))$ solves the continuity equation \eqref{eq:contcont}.
\end{enumerate}

\end{proposition}
\begin{proof}
These statements are direct consequences of the definitions and the Gauss--Green Theorem.
\end{proof}

It follows from the definitions that $\PN\circ \QN$ is the identity operator on $\prob{(\LN)}$. On the other hand, $\QN\circ\PN$ is a good approximation of the identity in the following sense.

\begin{lemma}\label{lem:QP}
For all $\mu\in\prob(\I)$ and all $N \geq 2$ we have
\begin{equation}
\label{eq:qnpn}
W_2(\QN(\PN(\mu)),\mu)\leq\frac{\sqrt{d}}{N}\;.
\end{equation}
\end{lemma}

\begin{proof}
Since both measures agree on each cube $\QA$, it follows that 
\begin{align*}
 W_2(\QN(\PN(\mu)),\mu)^2 
   \leq \sum_{\aa \in \LN} \mu(\QA) \diam(\QA)^2\;.
\end{align*}
Taking into account that the diameter of each $\QA$ equals ${\sqrt{d}}/{N}$, the result follows.
\end{proof}

The following simple result allows us to compare the $2$-Wasserstein distances on $\cP(\I)$ and $\cP(\LN)$. Recall that $W_{2,N}$ has been defined in  \eqref{eq:W2N}.

\begin{lemma}\label{lem:Wass-comparision}
For all $\mu_0, \mu_1 \in \cP(\I)$ we have
\begin{align*}
 W_{2,N}( \PN (\mu_0), \PN(\mu_1) ) \leq
 W_2(\mu_0, \mu_1)  + \frac{\sqrt{d}}{N}\;.
\end{align*}
\end{lemma}

\begin{proof}
Define $T_N : \I \to \LN$ by $T_N(x) := \aa$ whenever $x \in \QA$.
Since
 $|(T_N x)_i - (T_N y)_i| \leq 1 + N|x_i - y_i|$ for $x,y \in \I$, we have
\begin{align*}
 \dist(T_N x, T_N y) \leq 
 | x - y | + \frac{\sqrt{d}}{N}\;.
\end{align*}
Using the fact that $\PN(\mu_i) = (T_N)_\# \mu_i$, the result follows.
\end{proof}

In order to carry out our estimates, we will sometimes need some regularity on the probability densities involved. For this reason, we introduce the following set.
\begin{definition}[Regular densities]
Let $\delta>0$. Then the set $\D_\delta(\LN)\subset \D(\LN)$ is the set of probability densities $\rho_N\in\D(\LN)$ such that
\begin{align*}
\min_{\aa\in\LN}\rho_N(\aa)\geq \delta\;,\qquad\qquad\Lip_N(\rho_N)\leq \delta^{-1}\;.
\end{align*}
\end{definition}
Notice that the projections $\PN$ preserve this sort of regularity, i.e.,
\begin{equation}
\label{eq:pnreg}
\Lip_N(\PN(\rho))\leq \Lip(\rho)\;,\qquad \qquad\min_{\aa\in\LN}\PN(\rho)(\aa)\geq \inf_{x\in\I}\rho(x)\;,
\end{equation}
as is readily checked from the definitions.

The set $\D_\delta(\LN)$ is endowed with the following distance, which is obtained by minimizing the action functional over all paths in the space of regular densities.

\begin{definition}[The distance $\mathcal W_{N,\delta}$]\label{def:wdelta}
Let $\delta>0$ and $\rho_{N,0},\rho_{N,1}\in \D_\delta(\LN)$. The distance $\mathcal W_{N,\delta}(\rho_{N,0},\rho_{N,1})$ is defined as
\[
\big(\mathcal W_{N,\delta}(\rho_{N,0},\rho_{N,1})\big)^2:=\inf
\bigg\{\int_0^1\cA_N(\rho_{N,t},V_{N,t})\dd t\bigg\}\;,
\]
the infimum being taken among all solutions $(\rho_{N,t},V_{N,t})$ of the continuity equation \eqref{eq:contn} such that $\rho_{N,t}\in\D_{\delta}(\LN)$ for any $t\in[0,1]$.
\end{definition}

The last tool that we need is a variant of the distance $\cWN$ on $\cP(\LN)$, where instead of the logarithmic mean $\theta$ one considers the harmonic mean $\tilde\theta$ given by 
\begin{align*}
\tilde\theta(a,b):=\frac{2ab}{a+b}
\end{align*}
for any $a,b> 0$. If $a = 0$ or $b = 0$, we set $\tilde\theta(a,b) = 0$. 
For $\rho_N\in\D(\LN)$ and $R =  R_{\aa,i+}^N\in\RN$ we put 
\begin{align*}
\tilde\rho_N(R):=\tilde\theta(\rho_{N}(\aa),\rho_N(\aa+\ei))\;.
\end{align*}

\begin{definition}[The distance ${\widetilde{\mathcal W}}_{N}$]\label{def:wtilde}

For $\rho_{N,0},\rho_{N,1}\in\D(\LN)$, the metric ${\widetilde{\mathcal W}}_{N}(\rho_{N,0},\rho_{N,1})$ is defined as
\[
\big({\widetilde{\mathcal W}}_{N}(\rho_{N,0},\rho_{N,1})\big)^2:=\inf
\bigg\{\int_0^1 \frac{1}{4d^2 N^{d+2}}\sum_{R\in\RN}\frac{V_{N,t}(R)^2}{\tilde\rho_{N,t}(R)}\dd t\bigg\}\;,
\]
the infimum being taken among all solutions $(\rho_{N,t},V_{N,t})$ of the continuity equation \eqref{eq:contn}.
\end{definition}

Distances of this form have already been introduced in \cite{Ma11}.
Notice that since $\tilde\theta(a,b)\leq \theta(a,b)$ for any $a,b\geq 0$, it follows immediately that ${\widetilde{\mathcal W}}_{N}\geq \cWN.$

\bigskip

Let us now describe our strategy to prove Theorem \ref{thm:main}. We start with two measures $\mu_0,\mu_1\in\prob(\I)$, regularize them a bit using the heat flow for a short time $s>0$, and then show (Proposition \ref{prop:upper}) that for some constant $C(s) < \infty$ (independent on $\mu_0,\mu_1$) we have
\[
\cWN(\PN({\sfh}_s(\mu_0)), \PN({\sfh}_s(\mu_1)))
\leq  W_2(\mu_0, \mu_1) +  \frac{C(s)}{\sqrt{N}}\;.
\]
This will follow quite easily. The converse inequality will be harder to achieve, as the natural inequality that one obtains  for $\rho_{N,0},\rho_{N,1}\in\D(\LN)$ (in Proposition \ref{prop:lower}) involves the harmonic mean rather than the logarithmic mean, i.e., we prove that
\[
W_2(\QN(\rho^N_0),\QN(\rho^N_1))\leq \tWN(\rho^N_0,\rho^N_1)\;.
\]

Thus the  problem becomes to bound $\tWN$ from above in terms of $\cWN$ plus a small error. Unfortunately, the harmonic-logarithmic mean inequality $\tilde\theta(a,b)\leq  \theta(a,b)$ goes in the `wrong' direction, but the elementary inequality 
\begin{align*}
\frac{1}{\tilde\theta(a,b)} - \frac{1}{\theta(a,b)} 
  \leq \frac{(b-a)^2}{ab}\frac{1}{\tilde\theta(a,b)}
\end{align*}
that we establish in Proposition \ref{prop:wtildewdelta}, allows us to obtain an estimate for all regular densities, i.e., 
\[
\widetilde{\mathcal W}_{N}(\rho^N_0,\rho^N_1)\leq \left(1-\frac{1}{\delta^4N^2}\right)^{-\frac12}\mathcal W_{N,\delta}(\rho^N_0,\rho^N_1)\;.
\]
for $\rho_0^{N},\rho_1^{N}\in\D_\delta(\LN)$,

Thus at the end everything reduces to prove that $\mathcal W_{N,\delta}$ can be bounded above, up to a small error, by $\cWN$. Clearly, this is false without  some additional assumptions on the measures we want to interpolate. The idea is then to notice that the measures on the discrete torus that we  produced in our first step, using $\PN$ after an application of the heat flow,  belong to $\D_{\delta}(\LN)$ for some $\delta>0$. We then show in Proposition \ref{prop:hard}, which is technically the most involved, that given $\eps,\delta>0$, there exists $\bar\delta>0$ such that the bound
\[
\mathcal W_{N,\bar\delta}(\rho_{N,0},\rho_{N,1})\leq \cWN(\rho_{N,0},\rho_{N,1})+\eps
\] 
holds for any $\rho_{N,0},\rho_{N,1}\in\D_\delta(\LN)$. This will be enough to complete the argument.

\subsection{Estimates}
\label{subsec:estim}

Here we collect all the estimates that we need to implement the strategy outlined above.
We start by observing the effect of $\PN$ on the action of vector fields.
\begin{lemma}
\label{lem:boundkinPn}
Let $\mu=\rho\pi\in\prob({\I})$ be a probability measure and $V:\I\to\R^d$ a momentum vector field. Assume that both $\rho$ and $V$ are Lipschitz and that $\min\rho>0$. Put $\rho^N:=\PN(\mu)$ and $V^N:=\PN(V)$. Then, for any $N\geq 2$ we have the bound 
\begin{equation}
\label{eq:boundkinPn}
\cA_N(\rho^N,V^N)
 \leq \int_{\I}\frac{|V(x)|^2}{\rho(x)}\dd x+
  \frac{1}{N}
  \bigg( \frac{\|V\|_{L^\infty} \Lip(V)}{\min \rho} 
         +         \frac{(1 + \Lip(\rho))^2  }{(\min \rho)^3} 
\|V\|_{L^\infty}^2 
	 \bigg) 	 \;.
\end{equation}
\end{lemma}
\begin{proof}
We apply Jensen's inequality to the convex function $(x,y,z)\mapsto \frac{x^2}{\theta(y,z)}$ to obtain for $R = \RApm \in \RN$,
\begin{equation}\begin{aligned}
\label{eq:Jensen}
\frac{1}{4d^2 N^{d+2}}\frac{V^N(R)^2}{\hrho^N(R)}&   = \frac{1}{N^2}\frac{ \Big( \int_R  V_i(r) \dd r \Big)^2}
     	{\theta\Big( \int_R \int_0^{1/N} \rho(r - h \ei) \dd h \dd r \ , \  
			\int_R \int_0^{1/N} \rho(r + h \ei) \dd h \dd r\Big)}
 \\&  \leq 
 \int_R  \int_0^{\frac{1}{N}}  
 \frac{  |  V_i(r) |^2}
     	{\theta\big(\rho(r - h \ei)\ , \  
			\rho(r + h \ei) \big)}
			\dd h \dd r
 \\&  = \frac12
 \int_R  \int_{-\frac{1}{N}}^{\frac{1}{N}}  
 \frac{  |  V_i(r) |^2}
     	{\theta\big(\rho(r - h \ei)\ , \  
			\rho(r + h \ei) \big)}
			\dd h \dd r\;.
\end{aligned}\end{equation}
Since $\partial_a \theta(a,b) \leq \frac{\theta(a,b)}{a}$, we infer that
\begin{align*}
\big| \rho(r + h \ei) -  \theta\big(\rho(r - h \ei)\ , \  
			\rho(r + h \ei)  \big) \big|
&			\leq \frac{\max \rho }{\min \rho}  \;\big| 
			\rho(r + h \ei)  - \rho(r - h \ei)  \big|
\\&			\leq \frac{1 + \Lip(\rho)}{\min \rho}  \;
			\frac{2\Lip(\rho)}{N}\;.
\end{align*}
Combining this with the elementary fact that for $x, \tx \in \R$ and $y \geq  \ty > 0$,
\begin{align*}
 \Big|\frac{x^2}{y} - \frac{\tx^2}{\ty}\Big|
  \leq \frac{|x + \tx|}{\ty} | x - \tx| + \frac{x^2}{\ty^2} |y - \ty|\;,
\end{align*}
we obtain for $r \in R$ and $|h| \leq \frac{1}{N}$, 
\begin{align*}
\frac12\bigg| \frac{  |  V_i(r) |^2}
     	{\theta\big(\rho(r - h \ei)\ , \  
			\rho(r + h \ei) \big)}
			 -
 \frac{|V_i(r + h \ei)|^2}{\rho(r + h \ei)} \bigg|
&  \leq \frac{1}{N}			
    \bigg(  \frac{\|V\|_{L^\infty} \Lip(V)}{\min \rho} 
     \\&\qquad  \qquad  + 
         \frac{\|V\|_{L^\infty}^2  }{(\min \rho)^2} 
         \frac{1 + \Lip(\rho)}{\min \rho}  \;
			{\Lip(\rho)} \bigg) 			 \;.
\end{align*}
Combining this bound with \eqref{eq:Jensen}, and summing over all $R\in\RN$ the result follows.
\end{proof}

The previous result can be used to obtain the following lower bound for the Wasserstein metric $W_2$.

\begin{proposition}\label{prop:upper}
Let $s > 0$. There exists a dimensional constant $C(s) < \infty$ such that for all probability measures $\mu_0, \mu_1 \in \prob(\I)$ and for all $N \geq 1$  we have
\begin{align*}
 \cWN(\PN({\sfh}_s(\mu_0)), \PN({\sfh}_s(\mu_1)))
\leq  W_2(\mu_0, \mu_1) +  \frac{C(s)}{\sqrt{N}}\;.
\end{align*}
\end{proposition}
\begin{proof} Let $(\mu_t)$ be a constant speed geodesic connecting $\mu_0$ to $\mu_1$ in $(\prob(\I),W_2)$, and let $(v_t)$ denote the corresponding velocity vector field achieving the minimum in \eqref{eq:BeBr}. For $s > 0$, let $\rho_{s,t}$ and $V_{s,t}$ be the densities with respect to $\pi$ of ${\sfh}_s(\mu_t)$ and ${\sfh}_s(v_t\mu_t)$ respectively. According to $(iii)$  of Proposition \ref{prop:heatcont}, for given $s>0$, the curve $t\mapsto (\rho_{s,t},V_{s,t})$ is a solution to the continuity equation \eqref{eq:contcont} and we have
\begin{equation}
\label{eq:actiondec}
\int_0^1\int_{\I}\frac{|V_{s,t}(x)|^2}{\rho_{s,t}(x)}\dd t\dd x\leq W_2^2(\rho_0,\rho_1)\;.
\end{equation}
By $(i)$ and $(ii)$ of Proposition \ref{prop:heatcont} we also know that there exists constants $c(s) > 0$ and $C(s) < \infty$ such that for all $t \in [0,1]$,
\begin{align}\label{eq:rho-bounds}
\inf_{x\in \I}\rho_{s,t}(x)\geq c(s)\;, \quad
\Lip(\rho_{s,t}) \leq C(s)\;, \quad
 \, \|V_{s,t}\|_{L^\infty}  + \Lip(V_{s,t})
    \leq C(s) \|V_{s/2,t}\|_{L^1} \;.
\end{align}
Set $t\mapsto \eta_{N,t}:=\PN({\sfh}_s(\mu_t))$ and $t\mapsto W_{N,t}:=\PN(V_{s,t})$. By Proposition \ref{prop:solcont} the curve $(\eta_{N,t},W_{N,t})$ solves the continuity equation \eqref{eq:contn}. Applying Lemma \ref{lem:boundkinPn}, \eqref{eq:rho-bounds} and \eqref{eq:actiondec}, we obtain for some (different) constant $C(s) < \infty$, 
\[
\begin{split}
&\cWN(\PN({\sfh}_s(\mu_0)), \PN({\sfh}_s(\mu_1)))^2
\\&\leq \int_0^1\cA_N(\eta_{N,t},W_{N,t})\dd t\\
&\leq\int_0^1\bigg[\int_{\I}\frac{|V_{s,t}(x)|^2}{\rho_{s,t}(x)}\dd x+\frac{1}{N}
  \bigg(  \frac{\|V_{s,t}\|_{L^\infty} \Lip(V_{s,t})}{\min \rho_{s,t}} 
         +\frac{(1 + \Lip(\rho_{s,t}))^2}{(\min \rho_{s,t})^3} \|V_{s,t}\|_{L^\infty}^2  \bigg) \bigg] \dd t
\\&\leq
W_2^2(\rho_0,\rho_1) + \frac{C(s)}{N}
\int_0^1 \|V_{s/2,t}\|_{L^1}^2\dd t\;.
\end{split}
\]
Applying the Cauchy-Schwarz inequality in the form
\begin{align*}
  \left\|V_{s/2,t}\right\|_{L^1}^2 
   \leq \int_{\I} \frac{|V_{s/2,t}(x)|^2}{\rho_{s/2,t}(x)} \dd x\;,
\end{align*}
we obtain
\begin{align*}
\cWN(\PN({\sfh}_s(\mu_0)), \PN({\sfh}_s(\mu_1)))^2
 & \leq W_2(\rho_0,\rho_1)^2
     + \frac{C(s)}{N} \int_0^1 
\int_{\I} \frac{|V_{s/2,t}(x)|^2}{\rho_{s/2,t}(x)} \dd x\dd t
 \\&\leq  W_2(\rho_0,\rho_1)^2  + \frac{C(s)}{N} W_2(\rho_0,\rho_1)^2\;.
\end{align*}
Taking into account that $(\cP(\I),W_2)$ has finite diameter, we obtain the the result by taking square roots and using that $\sqrt{a+b}\leq\sqrt a+\sqrt b$.
\end{proof}

The next result provides a lower bound for $W_2$. Recall that $\tWN$ is defined using the harmonic mean instead of the logarithmic mean.

\begin{proposition}\label{prop:lower}
Let $N\geq 1$ and $\rho^N_0,\rho^N_1\in\mathscr P(\LN)$. Then
\begin{equation}
\label{eq:w2wt}
W_2(\QN(\rho^N_0),\QN(\rho^N_1))\leq \tWN(\rho^N_0,\rho^N_1)\;.
\end{equation}
\end{proposition}
\begin{proof}
Let  $t\mapsto (\rhon_t, \Vn_t)$ be a  solution to the  continuity equation \eqref{eq:contn}. Define $\rho_t := \QN (\rhon_t)$ and $V_t := \QN (\Vn_t)$. Then, for every $t \in [0,1]$ we have
\begin{align*}
 &\int_{\I} \frac{|V_t(x)|^2}{\rho_t(x)} \dd x  
      = \sum_{\aa \in \LN}  \int_{\QA} \frac{|V_t(x)|^2}{\rho_t(x)} \dd x  
 \\ & =    \frac{1}{N^{d-1}}\sum_{\aa, i} 
   \frac{1}{\rhon_t(\aa)} \int_{\frac{a_i}{N}}^{\frac{a_i + 1}{N}} {\bigg|\frac{1 - N x_i + a_i}{2dN}\Vn_t(\RAm) + 
    \frac{N x_i - a_i}{2dN} \Vn_t(\RAp)\bigg|^2}\dd x_i  
 \\ & =    \frac{1}{4d^2N^{d+2}}\sum_{\aa, i} 
   \frac{1}{\rhon_t(\aa)} \int_0^1 {\big|(1-y)\Vn_t(\RAm) + y \Vn_t(\RAp)\big|^2}\dd y  \\
& \leq
     \frac{1}{4d^2N^{d+2}}\sum_{\aa, i} 
   \frac{ \Vn_t(\RAm)^2 +  \Vn_t(\RAp)^2 }{2\rhon(\aa)}
\\&  = 
     \frac{1}{4d^2N^{d+2}}\sum_{\aa, i} \frac{\Vn_t(\RAp)^2}{2}
  \bigg( \frac{ 1}{\rhon_t(\aa)} + \frac{1}{\rhon_t(\aa + \ei)}\bigg)
\\&  =
     \frac{1}{4d^2N^{d+2}}\sum_{\aa, i}  \frac{\Vn_t(\RAp)^2}{\tilde\rho_t^N(\RAp)}\;.
\end{align*}
Since from Proposition \ref{prop:solcont} we know that  $t\mapsto(\rho_t, V_t)$ solves the continuity equation, we obtain
\begin{align*}
 W_2^2(\rho_0, \rho_1)
 & \leq  \int_0^1 \int_{\I} \frac{|V_t(x)|^2}{\rho_t(x)} \dd x  \dd t \leq    \frac{1}{4d^2N^{d+2}}\sum_{R \in \RN} 
  \int_0^1  \frac{\Vn_t(R)^2}{\tilde{\rho}^N_t(R)} \dd t\;.
  \end{align*}
Taking the infimum over all the solutions $(\rho^N_t, \Vn_t)$ of \eqref{eq:contn} and recalling the Definition \ref{def:wtilde} of $\tWN$ we get the result.
\end{proof}

For regular densities, the following result compares the distances defined using the harmonic and the logarithmic means. Note that the reverse inequality $\cWN \leq \tWN$ follows directly from the harmonic-logarithmic mean inequality. It is possible to obtain a better (i.e. larger) numerical constant  in the denominator appearing in \eqref{eq:wntdwnd}, but the stated estimate suffices for our purpose.

\begin{proposition}\label{prop:wtildewdelta}
Let $\delta>0$, $N > \delta^{-2}$ and $\rho^N_0\rho^N_1\in \D_{\delta}(\LN)$. Then the following estimate holds:
\begin{equation}
\label{eq:wntdwnd}
\widetilde{\mathcal W}_{N}(\rho^N_0,\rho^N_1)\leq \left(1-\frac{1}{\delta^4N^2}\right)^{-\frac12}\mathcal W_{N,\delta}(\rho^N_0,\rho^N_1)\;.
\end{equation}
\end{proposition}

\begin{proof}
Let $b\geq a>0$ and, as before, let $\tilde\theta(a,b):=\frac{2}{\frac1a+\frac1b}$ be the harmonic mean. Set $f(t) = ((1-t)a + t b)^{-1}$ and notice 
that
\begin{align*}
\frac{1}{\theta(a,b)} = \int_0^1 f(t) \dd t\;,\qquad
\frac{1}{\tilde\theta(a,b)} = \frac12 (f(0) + f(1))\;.
\end{align*}
Integrating by parts, and using that $f'(0) \leq f'(t) \leq f'(1)$ since $f$ is convex, we obtain
\begin{align*}
 \frac{1}{\tilde\theta(a,b)} - \frac{1}{\theta(a,b)} 
   &= \frac12 \int_0^1 f(0) - f(t) \dd t + 
       \frac12 \int_0^1 f(1) - f(t) \dd t
%\\&= \frac12 \int_0^1 (2t-1)f'(t) \dd t 
 \\&= \frac12 \int_0^1 (t-1) f'(t) \dd t + 
       \frac12 \int_0^1 t  f'(t) \dd t
 \\& \leq  \frac12 \Big(-  f'(0) + f'(1) \Big)
 = \frac{b-a}{2}\bigg(\frac{1}{a^2} - \frac{1}{b^2}\bigg)
 = \frac{(b-a)^2}{ab}\frac{1}{\tilde\theta(a,b)}\;.
\end{align*}
Therefore, for $\rho^N\in\cP_{\delta}(\LN)$ and $R \in \RN$ we have
\begin{align*}
\frac{1}{\tilde\rho^N(R)}\leq \left(1-\frac{1}{\delta^4N^2}\right)^{-1}\frac{1}{\hat\rho^N(R)}\;,
\end{align*}
and the result follows applying this inequality along a geodesic in $(\D_{ \delta}(\LN), \mathcal W_{N,\delta})$ connecting $\rho^N_0$ to $\rho^N_1$.
\end{proof}

The final proposition in this subsection shows that regular densities can be connected by a curve consisting of (a bit less) regular densities, for which the action functional is almost optimal. 

\begin{proposition}\label{prop:hard}
Let $\eps,\delta\in(0,1)$. Then there exists $\bar\delta>0$ such that  for any $N\geq 4$ and $\rho_{N,0},\rho_{N,1}\in \D_{\delta}(\LN)$, we have the bound
\begin{equation}
\label{eq:wndwn}
 \mathcal{W}_{N,\bar\delta}(\rho_{N,0},\rho_{N,1})\leq\cWN(\rho_{N,0},\rho_{N,1})+\eps\;.
\end{equation}
\end{proposition}
\begin{proof} Let $a,b\in(0,\delta)$ to be fixed later and $t\mapsto (\rho_{N,t},V_{N,t})$ be a $\cWN$-geodesic connecting $\rho_{N,0}$ to $\rho_{N,1}$. Define the curves $t\mapsto (\rho_{N,t}^1,V^1_{N,t})$  and $t\mapsto (\rho^2_{N,t},V^2_{N,t})$ by
\begin{align}
 \rho^1_{N,t}&:=(1-a)\rho_{N,t}+a\;,\qquad 
&V^1_{N,t}&:=(1-a)V_{N,t}\;,\\
\rho^2_{N,t}&:={\sfh}_{b}^N(\rho^1_{N,t})\;,
\qquad &\qquad \ V^2_{N,t}&:={\sfh}_{b}^N(V^1_{N,t})\;.
\end{align}
The latter expression should be interpreted in the sense of \eqref{eq:heat-vector}.

\noindent{\bf Step 1: From $\rho_{N,j}$ to $\rho^1_{N,j}$ for $j = 0,1$.}

For $j=0,1$, we  define $s\mapsto \eta_{N,s,j}$ as the linear interpolation between $\rho_{N,j}$ and $\rho^1_{N,j}$, i.e., 
\begin{align*}
\eta_{N,s,j}(\aa):=(1-s)\rho_{N,j}(\aa)+s\rho^1_{N,j}(\aa)=\rho_{N,j}(\aa)+sa\big(1-\rho_{N,j}(\aa)\big)\;.
\end{align*}
Notice that since $\sum_{\aa\in\LN}1-\rho_{N,j}(\aa)=0$, it makes sense to define
\begin{align*}
W_{N,s,j}(R^N_{\aa,i\pm}):=\mp2a dN^2\Big( \Delta^{-1}_N(\one-\rho_{N,j})(\aa\pm\ei)-\Delta^{-1}_N(\one-\rho_{N,j})(\aa)\Big)\;,
\end{align*}
with $\one$ being the density constantly equal to one. A direct computation shows that $s\mapsto(\eta_{N,s,j},W_{N,s,j})$ is a solution to the continuity equation \eqref{eq:contn}. Notice that actually $W_{N,s,j}$ does not depend on $s$.
Taking into account that 
\begin{equation}
\label{eq:belowstep1}
\eta_{N,s,j}(\aa)\geq a\;,\qquad \aa\in\LN\;,\ s\in[0,1]\;,\ j=0,1\;,
\end{equation}
recalling the Poincar\'e inequality (Proposition \ref{prop:Poincare}),
and using the trivial bound
\begin{align*}
\cEN(1-\rho_{N,j}) \leq  d\,(\Lip_N(\rho_{N,j}))^2 \leq d{\delta^{-2}}\;,
\end{align*}
we obtain 
\begin{equation}
\label{eq:step1}
\begin{split}
\cA_N(\eta_{N,s,j},W_{N,s,j})
&=\frac{1}{4d^2N^{d+2}}\sum_{\aa\in\LN}\sum_{i=1}^d\frac{\big(W_{N,s,j}(R^N_{\aa,i+})\big)^2}{\hat\eta_{N,s,j}(R^N_{\aa,i+})}\\
&\leq\frac{a}{N^{d-2}}\sum_{\aa\in\LN}\sum_{i=1}^d
 \Big( \Delta^{-1}_N(\one-\rho_{N,j})(\aa+\ei)-\Delta^{-1}_N(\one-\rho_{N,j})(\aa)\Big)^2\\
&=a\, \mathcal E_N(\Delta^{-1}_N(\one-\rho_{N,j}))\\
&\leq \frac{a}{\kappa}\|\one-\rho_{N,j}\|_{L^2_N}^2\\
&\leq \frac{a}{\kappa^2} \cEN(\one-\rho_{N,j})\\
&\leq \frac{ad}{\kappa^2 \delta^2}\;,
\end{split}
\end{equation}
where $\kappa:=\inf_{N\geq 4}2N^2(1-\cos(2\pi/N))>0$. 
Notice also that
\begin{equation}
\label{eq:lipstep1}
\Lip_N(\eta_{N,s,j})\leq \Lip_N(\rho_{N,j})\leq \delta^{-1}\;,\qquad s\in[0,1]\;,\ j=0,1\;.
\end{equation}

\noindent{\bf Step 2: From  $\rho^1_{N,j}$ to $\rho^2_{N,j}$ for $j = 0,1$.}

For $j=0,1$ we interpolate from  $\rho^1_{N,j}$ and $\rho^2_{N,j}$ using the heat flow, i.e., we define $s\mapsto (\sigma_{N,s,j},Z_{N,s,j})$ by
\[
\begin{split}
\sigma_{N,s,j}(\aa)&:={\sfh}_{sb}^N(\rho^1_{N,j})\;,\\
Z_{N,s,j}(R^N_{\aa,i\pm})&:= \mp 2b dN^2\big(\sigma_{N,s,j}(\aa\pm\ei)-\sigma_{N,s,j}(\aa)\big)\;.
\end{split}
\]
We then obtain
\[
\begin{split}
&\cA_N(\sigma_{N,s,j},Z_{N,s,j})\\
&=\frac{1}{4d^2N^{d+2}}\sum_{\aa\in\LN}\sum_{i=1}^d\frac{Z_{N,s,j}(R^N_{\aa,i+})^2}{\hat\sigma_{N,s,j}(R^N_{\aa,i+})}\\
&=\frac{b^2}{N^{d-2}}\sum_{\aa,i}\frac{\big(\sigma_{N,s,j}(\aa+\ei)-\sigma_{N,s,j}(\aa)\big)^2}{\hat\sigma_{N,s,j}(R^N_{\aa,i+})}\\
&=\frac{b^2}{N^{d-2}}\sum_{\aa,i}\big(\sigma_{N,s,j}(\aa+\ei)-\sigma_{N,s,j}(\aa)\big)\big(\log(\sigma_{N,s,j}(\aa+\ei))-\log(\sigma_{N,s,j}(\aa))\big)\\
&=b^2\mathcal E_N\big(\sigma_{N,s,j},\log(\sigma_{N,s,j})\big)\;.
\end{split}
\]
In view of Proposition \ref{prop:heatdisc}$(i)$ we obtain by construction,
\begin{align}
\label{eq:step2bounds1}
\sigma_{N,s,j}(\aa)&\geq \delta\;,\qquad&& \aa\in\LN,\ s\in[0,1],\ j=0,1\;,\\\label{eq:step2bounds2}
\Lip_N(\sigma_{N,s,j})&\leq \Lip_N(\rho_{N,j}^1)\leq \delta^{-1}\;,\qquad&& s\in[0,1],\ j=0,1\;.
\end{align}
Hence $\Lip_N(\log(\sigma_{N,s,j}))\leq \frac{\Lip_N(\sigma_{N,s,j})}{\min\sigma_{N,s,j}}\leq \delta^{-2}$. Since $|\cEN(f,g)| \leq d \Lip_N(f)\Lip_N(g)$ we obtain
\begin{equation}
\label{eq:step2}
\cA_N(\sigma_{N,s,j},Z_{N,s,j})\leq \frac{d b^2}{\delta^3}\;.
\end{equation}

\noindent{\bf Step 3: From $\rho^2_{N,0}$ to $\rho^2_{N,1}$.} 

From the convexity of the function $(x,a,b)\mapsto\frac{x^2}{\theta(a,b)}$  we get 
\[
\cA_N(\rho^1_{N,t},V^1_{N,t})\leq (1-a)\cA_N(\rho_{N,t},V_{N,t})=(1-a)\cWN(\rho_{N,0},\rho_{N,1})^2\;,
\]
for any $t\in[0,1]$. Using again the convexity of $(x,a,b)\mapsto\frac{x^2}{\theta(a,b)}$  and the fact that ${\sfh}$ acts as a convolution semigroup, we also get
\[
\cA_N(\rho^2_{N,t},V^2_{N,t})\leq \cA_N(\rho^1_{N,t},V^1_{N,t})
\]
for any $t\in[0,1]$.
Combining these two inequalities and integrating we get
\begin{equation}
\label{eq:step3}
\int_0^1\cA_N(\rho^2_{N,t},V^2_{N,t})\dd t\leq \int_0^1 \cA_N(\rho^1_{N,t},V^1_{N,t})\leq (1-a)\cWN(\rho_{N,0},\rho_{N,1})^2\;.
\end{equation}
Since the heat semigroup preserves positivity, we obtain
\begin{equation}
\label{eq:step3bound}
\rho^2_{N,t}(\aa)\geq a\;,\qquad \aa\in\LN,\ t\in[0,1]\;,
\end{equation}
and by $(i)$ of Proposition \ref{prop:heatdisc} we have
\begin{equation}
\label{eq:step3lip}
\Lip_N(\rho^2_{N,t})\leq C(b)\;,\qquad t\in[0,1]\;,
\end{equation}
for some constant $C(b) > 0$ which depends only on $b$ and on the dimension $d$.
 
\noindent{\bf Step 4: Gluing the pieces.} 

Let $\ell\in(0,1/4)$ to be fixed later. We define the curve $t\mapsto (\rho^3_{N,t},V^3_{N,t})$ on $[0,1]$ by gluing the pieces together, that is, 
\[
\begin{split}
 &(\rho^3_{N,t},V^3_{N,t}):=
 \left\{
 \begin{array}{lll}
 (\eta_{N,\frac{t}{\ell},0}&,\ell^{-1}W_{N,\frac{t}{\ell},0})&\qquad t\in[0,\ell]\;,\\
 (\sigma_{N,\frac{t-\ell}{\ell},0}&,\ell^{-1}Z_{N,\frac{t-\ell}{\ell},0})&\qquad t\in(\ell,2\ell)\;,\\
 (\rho^2_{N,\frac {t-2\ell}{1-4\ell}}&,(1-4\ell)^{-1}
  V^2_{N,\frac {t- 2\ell}{1-4\ell}})&\qquad t\in[2\ell,1-2\ell]\;,\\
 (\sigma_{N,\frac{1-\ell-t}{\ell},1}&,\ell^{-1}Z_{N,\frac{1-\ell-t}{\ell},1})&\qquad t\in(1-2\ell,1-\ell)\;,\\ 
 (\eta_{N,\frac{1-t}{\ell},1}&,\ell^{-1}W_{N,\frac{1-t}{\ell},1})&\qquad t\in[1-\ell,1]\;.
 \end{array}
 \right.
 \end{split}
\]
Clearly, $t\mapsto (\rho^3_{N,t},V^3_{N,t})$ is a solution to the continuity equation \eqref{eq:contn}. From \eqref{eq:step1}, \eqref{eq:step2} and \eqref{eq:step3} we get, taking the scaling factors into account,
\begin{align*}
\int_0^1 \cA_N(\rho^3_{N,t},V^3_{N,t})
  \leq\frac{2ad}{\ell\kappa^2\delta^2}
   +\frac{2d b^2}{\ell\delta^3}
   +\frac{ 1-a}{1-4\ell}\cWN(\rho_{N,0},\rho_{N,1})^2\;.
\end{align*}
It remains to fix the constants $a,b\in(0,\delta)$ and $\ell\in(0,1/4)$ as functions of $\delta$ and $\eps$. From (ii) of Proposition \ref{prop:uniform} we know that the diameter of $(\D(\LN), \cWN)$ is bounded by a constant $D > 0$ depending only on $d$.
Choose now $\ell > 0$ so small that $\frac{1}{1-4\ell} \leq 1+\frac{\eps^2}{3D^2}$, and then $a, b > 0$ so small that   
\[
\frac{2ad}{\ell\kappa^2\delta^2} \leq \frac{\eps^2}3\;, \qquad 
\frac{2d b^2}{\ell\delta^3} \leq \frac{\eps^2}3\;.
\]
With these choices we get
\begin{equation}
\label{eq:finalest}
\int_0^1 \cA_N(\rho^3_{N,t},V^3_{N,t})\leq\eps^2+\cWN(\rho_{N,0},\rho_{N,1})^2\;.
\end{equation}
Furthermore, the inequalities \eqref{eq:belowstep1}, \eqref{eq:step2bounds1}, and \eqref{eq:step3bound} and the inequalities  \eqref{eq:lipstep1}, \eqref{eq:step2bounds2} and \eqref{eq:step3lip} imply that 
\begin{align*}
 \min \rho^3_{N,t} \geq a\;, \qquad
  \Lipn(\rho^3_{N,t})\leq \max\{ \delta^{-1},  {C}(b) \}\;,
\end{align*}
hence $\rho^3_{N,t}$ belongs to $\prob_{\bar\delta}(\LN)$ for some $\bar\delta$ depending on $a,b$ and $\delta$.
 The result follows in view of Definition \ref{def:wdelta} of $\mathcal W_{N,\bar\delta}$.
\end{proof}

\subsection{Wrap up and conclusion of the argument}
\label{subsec:wrapup}

Finally we shall prove Theorem \ref{thm:main}. Let us first recall one of the equivalent characterisations of Gromov-Hausdorff convergence, which we formulate here as a definition. We refer to, e.g., \cite[Definition 27.6 and (27.4)] {Vil09}) for more details. 

\begin{definition}[Gromov-Hausdorff Convergence]\label{def:GH}
We say that a sequence of compact metric spaces $(\cX_n,d_n)$ converges in the sense of Gromov-Hausdorff to a compact  metric space $(\cX,d)$, if there exists a sequence of maps $f_n : \cX\to \cX_n $ which are 
\begin{enumerate}[$(i)$]
\item $\eps_n$-isometric, i.e., for all $x, y \in \cX$,
\begin{align*}
  |d_n(f_n(x), f_n(y)) - d(x,y)| \leq \eps_n\;; \qquad\qquad \text{and }
\end{align*}
\item $\eps_n$-surjective, i.e., for all $x_n \in \cX_n$ there exists $x \in \cX$ with 
\begin{align*}
d(f_n(x),x_n) \leq \eps_n\;,
\end{align*}
\end{enumerate}
for some sequence $\eps_n \to 0$.
\end{definition}

Now we are ready to prove our main result Theorem \ref{thm:main}, which we restate for the convenience of the reader.

\begin{theorem*}
Let $d \geq 1$. Then the metric spaces $(\prob(\LN),\mathcal W_N)$ converge to $(\prob(\I),W_2)$ in the sense of Gromov-Hausdorff as $N \to \infty$.
\end{theorem*}

\begin{proof}
For $s>0$ and $N \geq 1$ we consider the map from $\mathscr P(\I)$ to $\mathscr P(\LN)$ given by
\[
\mu\ \mapsto\ \PN({\sfh}_s \mu)\;.
\]
We claim that for each $s > 0$ there exists $\bar N(s) \geq 1$ such that for all $N \geq \bar N(s)$ this map is both $\eps(s)$-isometric and $\eps(s)$-surjective, for some sequence $\eps(s)\downarrow 0$ as $s \downarrow 0$.
This suffices to prove the theorem.

\medskip

\noindent{\bf $\eps(s)$-isometry}. 
Let $\mu_0, \mu_1 \in \cP(\I)$. Part $(i)$ of Proposition \ref{prop:heatcont} in conjunction with \eqref{eq:pnreg} yields that $\PN({\sfh}_s\mu_0)$ and $\PN({\sfh}_s\mu_1)$ belong to $\D_{\delta(s)}(\LN)$ for some $\delta(s)>0$ and for any $N\geq 1$. Let $\eta > 0$. From Proposition \ref{prop:hard} we then get the existence of $\bar\delta(\eta,s) > 0$ such that
\begin{align*}
 \mathcal W_{N,\bar\delta(\eta,s)}\big(\PN({\sfh}_s\mu_0), \PN({\sfh}_s \mu_1)\big) \leq \cWN\big(\PN({\sfh}_s\mu_0), \PN({\sfh}_s \mu_1)\big)+\eta\;.
\end{align*}
From Proposition \ref{prop:wtildewdelta} we infer that
\begin{align*}
\widetilde{\mathcal W}_{N}&\big(\PN({\sfh}_s\mu_0), \PN({\sfh}_s\mu_1)\big)\leq\left(1-\frac{1}{\bar\delta(\eta,s)^4N^2}\right)^{-\frac12} \mathcal W_{N,\bar\delta(\eta,s)}\big(\PN({\sfh}_s\mu_0), \PN({\sfh}_s\mu_1)\big)\;,
\end{align*}
and then from Proposition \ref{prop:lower} that
\begin{align*}
W_2\big(\QN(\PN({\sfh}_s\mu_0)),\QN(\PN({\sfh}_s\mu_1))\big)
\leq \widetilde{\mathcal W}_{N}\big(\PN({\sfh}_s\mu_0), \PN({\sfh}_s\mu_1)\big)\;.
\end{align*}
Lemma \ref{lem:QP} and Proposition \ref{prop:heatcont}$(i)$ yield
\begin{align*}
W_2(\mu_0,\mu_1)
 \leq 
W_2\big(\QN (\PN({\sfh}_s\mu_0)),\QN(\PN({\sfh}_s\mu_1))\big)
 + 2C\sqrt{s} + 2\frac{\sqrt{d}}{N} \;.
\end{align*}
Combining these four inequalities, we obtain
\begin{align*}
W_2(\mu_0,\mu_1)\leq \left(1-\frac{1}{\bar\delta(\eta,s)^4N^2}\right)^{-\frac12}\Big(\cWN\big(\PN({\sfh}_s\mu_0), \PN({\sfh}_s\mu_1)\big)+\eta\Big) + 2C\sqrt{s} + 2\frac{\sqrt{d}}{N}\;.
\end{align*}
On the other hand, Proposition \ref{prop:upper} grants that 
\begin{align*}
\cWN\big(\PN({\sfh}_s\mu_0), \PN({\sfh}_s\mu_1)\big)\leq  W_2(\mu_0, \mu_1) +  \frac{C(s)}{\sqrt{N}}\;.
\end{align*} 
Taking Proposition \ref{prop:uniform}(ii) into account, the latter two inequalities yield that for $\bar N = \bar N(s)$ sufficiently large and $\eta = \eta(s)$ sufficiently small, we have for all $N \geq \bar N(s)$,
\begin{align*}
\Big|W_2(\mu_0,\mu_1)- \cWN\big(\PN({\sfh}_s\mu_0), \PN({\sfh}_s\mu_1)\big)\Big|\leq \eps(s)
\end{align*}
for some $\eps(s)\downarrow 0$ as $s\downarrow0$.
\medskip

\noindent{\bf $\eps(s)$-surjectivity}. 
Let $\rhon \in \cP(\LN)$ and set $\rhon_s :=  {\sfh}_s \QN (\rhon)$. Then, for some dimensional constant $C < \infty$ which may change from line to line,
we obtain using Proposition \ref{prop:uniform}$(ii)$, Lemma \ref{lem:Wass-comparision}, and Proposition \ref{prop:heatcont}$(i)$, %
\begin{align*}
 \cWN\big(\rhon, \PN(\rhon_s)\big)
   & = \cWN\big(\PN (\QN(\rhon)), \PN(\rhon_s) \big)
 \\& \leq C W_{2,N}\big(\PN (\QN (\rhon)), \PN(\rhon_s) \big)
 \\& \leq C W_{2}\big(\QN (\rhon), \rhon_s \big) + \frac{C}{N}
 \\& \leq C \Big( \sqrt{s} + \frac{1}{N}\Big)\;.
\end{align*}
Taking, say, $N = 1/\sqrt{s}$, we infer that $\PN \circ {\sfh}_s$ is $2C\sqrt{s}$-surjective, which completes the proof.
\end{proof}

\bibliographystyle{plain}
\bibliography{GH-conv}

 \end{document}